\documentclass[12pt]{amsart}
\usepackage{amsmath,amssymb,enumerate}
\usepackage[pdftex,pagebackref=false]{hyperref} 
\newtheorem{theorem}{Theorem}[section]
\newtheorem{claim}{}[theorem]
\newtheorem{lemma}[theorem]{Lemma}

\newtheorem{corollary}[theorem]{Corollary}
\newtheorem{conjecture}[theorem]{Conjecture}
\theoremstyle{definition}

\newcommand{\cF}{\mathcal{F}}

\newcommand{\cM}{\mathcal{M}}

\newcommand{\cU}{\mathcal{U}}

\DeclareMathOperator{\si}{si}

\DeclareMathOperator{\cl}{cl}
\DeclareMathOperator{\PG}{PG}
\DeclareMathOperator{\GF}{GF}

\DeclareMathOperator{\wt}{wt}
\newcommand{\elem}{\epsilon}
\newcommand{\del}{\!\setminus\!}
\newcommand{\con}{/}
\newcommand{\dcon}{\con}

\newcommand{\fU}{\cU}

\hypersetup{
    plainpages=false,       
    unicode=false,          
    pdffitwindow=false,     
    pdfstartview={FitH},    
    pdftitle={uWaterloo\ LaTeX\ Thesis\ Template},   
    pdfnewwindow=true,      
    colorlinks=true,        
    linkcolor=black,         
    citecolor=black,        
}

\title[Exponentially Dense Matroids]{Projective Geometries in Exponentially Dense Matroids. II}
\author[Nelson]{Peter Nelson}	
\address{School of Mathematics, Statistics and Operations Research, Victoria University of Wellington, Wellington, New Zealand}

\begin{document}

	\maketitle

	\begin{abstract}
	We show for each positive integer $a$ that, if $\cM$ is a minor-closed class of matroids not containing all rank-$(a+1)$ uniform matroids, then there exists an integer $c$ such that either every rank-$r$ matroid in $\cM$ can be covered by at most $r^c$ rank-$a$ sets, or $\cM$ contains the $\GF(q)$-representable matroids for some prime power $q$ and every rank-$r$ matroid in $\cM$ can be covered by at most $cq^r$ rank-$a$ sets. In the latter case, this determines the maximum density of matroids in $\cM$ up to a constant factor. 
	\end{abstract}

	\section{Introduction}
	
	If $M$ is a matroid and $a$ is a positive integer, then $\tau_a(M)$ denotes the \emph{$a$-covering number} of $M$, the minimum number of sets of rank at most $a$ in $M$ required to cover $E(M)$. We will prove the following theorem:

	\begin{theorem}\label{main}
		Let $a \ge 1$ be an integer. If $\cM$ is a minor-closed class of matroids, then there is an integer $c > 0$ such that either
		\begin{enumerate}
			\item\label{m1} $\tau_a(M) \le r(M)^{c}$ for all $M \in \cM$,
			\item\label{m2} there is a prime power $q$ so that $\tau_a(M) \le c q^{r(M)}$ for all $M \in \cM$ and $\cM$ contains all $\GF(q)$-representable matroids, or
			\item\label{m3} $\cM$ contains all rank-$(a+1)$ uniform matroids. 
		\end{enumerate}
	\end{theorem}

	 This theorem also appears in [\ref{thesis}], and a weaker version, where the upper bound in (\ref{m2}) is replaced by $r(M)^{c}q^{r(M)}$, was proved in [\ref{part1}]; our proof is built with this weaker result as a starting point. $\tau_1(M)$ is just the number of points in $M$, and the above theorem was shown in this case by Geelen and Kabell [\ref{gk}].
	
	Theorem~\ref{main} resolves the `polynomial-exponential' part of the following conjecture of Geelen [\ref{openprobs}]:
	
	\begin{conjecture}[Growth Rate Conjecture]\label{grc}
		Let $a \ge 1$ be an integer. If $\cM$ is a minor-closed class of matroids, then there is an integer $c > 0$ so that either
		\begin{enumerate}
			\item $\tau_a(M) \le c r(M)$ for all $M \in \cM$, 
			\item $\tau_a(M) \le c r(M)^2$ for all $M \in \cM$ and $\cM$ contains all graphic matroids or all bicircular matroids, 
			\item there is a prime power $q$ so that $\tau_a(M) \le c q^{r(M)}$ for all $M \in \cM$ 
				and $\cM$ contains all $\GF(q)$-representable matroids, or
			\item\label{mciv} $\cM$ contains all rank-$(a+1)$ uniform matroids. 
		\end{enumerate}
	\end{conjecture}
	
	This conjecture was proved for $a = 1$ by Geelen, Kabell, Kung and Whittle~[\ref{gk},\ref{gkw},\ref{gw}] and is known as the `Growth Rate Theorem'.

	If (\ref{mciv}) holds, then $\tau_a(M)$ is not bounded by any function of $r(M)$ for all $M \in \cM$, as a rank-$(a+1)$ uniform matroid (and consequently any matroid with such a minor) can require arbitrarily many rank-$a$ sets to cover. Our bounds on $\tau_a$ are thus given with respect to some particular rank-$(a+1)$ uniform minor that is excluded. We prove Theorem~\ref{main} as a consequence of the two theorems below; the first is proved in [\ref{part1}], and the second is the main technical result of this paper. 
	
	\begin{theorem}\label{mainpoly}
		For all integers $a,b,n$ with $n \ge 1$ and $1 \le a < b$, there is an integer $m$ such that, if $M$ is a matroid of rank at least $2$ with no $U_{a+1,b}$-minor and $\tau_a(M) \ge r(M)^m$, then $M$ has a rank-$n$ projective geometry minor. 
	\end{theorem}
	
	\begin{theorem}\label{mainexp}
		For all integers $a,b,n,q$ with $n \ge 1$, $q \ge 2$ and $1 \le a < b$, there is an integer $c$ such that, if $M$ is a matroid with no $U_{a+1,b}$-minor and $\tau_a(M) \ge cq^{r(M)}$, then $M$ has a rank-$n$ projective geometry minor over a finite field with more than $q$ elements.
	\end{theorem}
		
	\section{Preliminaries}
	
	We use the notation of Oxley [\ref{oxley}]. A rank-$1$ flat is a \emph{point}, and a rank-$2$ flat is a \emph{line}. If $M$ is a matroid, and $X,Y \subseteq E(M)$, then $\sqcap_M(X,Y)$ denotes the \emph{local connectivity} between $X$ and $Y$ in $M$, defined by $\sqcap_M(X,Y) = r_M(X) + r_M(Y) - r_M(X \cup Y)$. If $\sqcap_M(X,Y) = 0$, then $X$ and $Y$ are \emph{skew} in $M$. Additionally, we write $\elem(M)$ for $\tau_1(M)$, the number of points in a matroid $M$. 
	
	For integers $a$ and $b$ with $1 \le a < b$, we write $\cU(a,b)$ for the class of matroids with no $U_{a+1,b}$-minor. The first tool in our proof is a theorem of Geelen and Kabell [\ref{gkb}] which shows that $\tau_a$ is bounded as a function of rank across $\cU(a,b)$. 
	
	\begin{theorem}\label{kdensity}
			Let $a$ and $b$ be integers with $1 \le a < b$. If $M \in \cU(a,b)$ and $r(M) > a$,
			then $\tau_a(M) \le \binom{b-1}{a}^{r(M)-a}$. 
		\end{theorem}
		\begin{proof}
			We first prove the result when $r(M) = a+1$, then proceed by induction. If $r(M) = a+1$, then observe that 
			$M|B \cong U_{a+1,a+1}$ for any basis $B$ of $M$; let $X \subseteq E(M)$ be maximal such that
			$M|X \cong U_{a+1,|X|}$. We may assume that $|X| < b$, and by maximality of $X$, every $e \in E(M)-X$ 
			is spanned by a rank-$a$ set of $X$. Therefore, $\tau_a(M) \le \binom{|X|}{a} \le \binom{b-1}{a}$. 
			
			Suppose that $r(M) > a+1$, and inductively assume that the result holds for matroids of smaller rank. 
			Let $e \in E(M)$. We have $\tau_{a+1}(M) \le \tau_a(M \con e) \le \binom{b-1}{a}^{r(M)-a-1}$ by induction, 
			and by the base case each rank-$(a+1)$ set in $M$ admits a cover with at most $\binom{b-1}{a}$ sets of 
			rank at most $a$. Therefore $\tau_a(M) \le \binom{b-1}{a} \tau_{a+1}(M) \le \binom{b-1}{a}^{r(M)-a}$, 
			as required. 
		\end{proof}
	
	The base case of this theorem gives $\tau_a(M) \le \binom{b-1}{a}\tau_a(M \con e)$ for all $M \in \cU(a,b)$ and $e \in E(M)$; an inductive argument yields the following:
	\begin{corollary}\label{kdensitycon}
		Let $a$ and $b$ be integers with $1 \le a < b$. If $M \in \cU(a,b)$ and $C \subseteq E(M)$, then $\tau_a(M \con C) \ge \binom{b-1}{a}^{-r_M(C)}\tau_a(M)$. 
	\end{corollary}
	Our starting point in our proof is the main technical result of [\ref{part1}]. Note that this theorem gives Theorem~\ref{mainpoly} when $q = 1$. 
	
	\begin{theorem}\label{halfwaypoint}
		There is an integer-valued function $f_{\ref{halfwaypoint}}(a,b,n,q)$ so that, for any integers $1 \le a < b$, $q \ge 1$ and $n \ge 1$, if $M \in \cU(a,b)$ satisfies $r(M) > 1$ and  $\tau_a(M) \ge r(M)^{f_{\ref{halfwaypoint}}(a,b,n,q)}q^{r(M)}$, then $M$ has a $\PG(n-1,q')$-minor for some prime power $q' > q$. 
	\end{theorem}
		
		\section{Stacks}
	
	We now define an obstruction to $\GF(q)$-representability. If $q$ is a prime power and $h$ and $t$ are nonnegative integers, then a matroid $S$ is a \emph{$(q,h,t)$-stack} if there are pairwise disjoint subsets $F_1, F_2, \dotsc, F_h$ of $E(S)$ such that the union of the $F_i$ is spanning in $S$, and for each $i \in \{1, \dotsc, h\}$ the matroid $(S \con (F_1 \cup \dotsc \cup F_{i-1}))|F_i$ has rank at most $t$ and is not $\GF(q)$-representable. We write $F_i(S)$ for $F_i$, and when the value of $t$ is unimportant, we refer simply to a \emph{$(q,h)$-stack}.
	
	 Note that a stack has rank between $2h$ and $th$, and that contracting or restricting to the sets in some initial segment of $F_1, \dotsc, F_h$ yields a smaller stack; we use these facts freely. 
	 
	We now show that the structure of a stack cannot be completely destroyed by a small projection. The following two lemmas are similar; the first does not control rank, and the second does. 
	
	\begin{lemma}\label{stackrobust}
			Let $q$ be a prime power, and $k \ge 0$ be an integer. If $M$ is a matroid, $C \subseteq E(M)$, and $M$ has a $(k(r_M(C)+1),q)$-stack restriction, then $(M \con C)|E(S)$ has a $(k,q)$-stack restriction. 
		\end{lemma}
		\begin{proof}
			Let $S$ be a $(k(r_M(C)+1),q)$-stack in $M$, with $F_i = F_i(S)$ for each $i$. By adding parallel extensions if needed, we may assume that $C \cap E(S)= \varnothing$. If $r_M(C) = 0$ then the result is trivial; suppose that $r_M(C) > 0$ and that the lemma holds for sets $C$ of smaller rank. Let $F = F_1 \cup \dotsc \cup F_k$.  If $C$ is skew to $F$ in $M$, then $(M \con C)|F$ is a $(k,q)$-stack, giving the lemma. Otherwise $M \con F$ has a $(k r_M(C),q)$-stack restriction, and $r_M(C) > r_{M \con F}(C)$. By the inductive hypothesis, $M \con (F \cup C)$ has a $(k,q)$-stack restriction $S'$; therefore $F \cup F_1(S'), F_2(S'), \dotsc, F_k(S')$ give a $(k,q)$-stack restriction of $M \con C$. 
		\end{proof}

	\begin{lemma}\label{skewstack}
		Let $q$ be a prime power, and $a,h$ and $t$ be integers with $a \ge 0$, $h \ge 1$ and $t \ge 2$. If $M$ is a matroid with an $((a+1)h,q,t)$-stack restriction $S$, and $X \subseteq E(M)$ is a set satisfying $\sqcap_M(X,E(S)) \le a$, then there exists $C \subseteq E(S)$ so that $(M \con C)|E(S)$ has an $(h,q,t)$-stack restriction $S'$, and $X$ and $E(S')$ are skew in $M \con C$.  
	\end{lemma}
	\begin{proof}
		Let $F = F_1(S) \cup \dotsc \cup F_h(S)$. If $F$ is skew to $X$ in $M$, then $F$ contains an $(h,q,t)$-stack $S'$ satisfying the lemma with $C = \varnothing$. Otherwise, $M \con F$ has an $(ah,q,t)$-stack restriction $S_0$ contained in $E(S)$, and $\sqcap_{M \con F}(X-F,E(S_0)) < \sqcap_{M}(X-F,E(S)) \le a$; the lemma follows routinely by induction on $a$.   
	\end{proof}
	
	This low local connectivity is obtained via the following lemma, which applies more generally. We will just use the case when $M|Y$ is a stack. 

	\begin{lemma}\label{reduceconn}
		If $M \in \fU(a,b)$ and $Y \subseteq E(M)$, then there is a set $X \subseteq E(M)$ so that $\tau_a(M|X) \ge \binom{b-1}{a}^{a-r_M(Y)}\tau_a(M)$ and $\sqcap_M(X,Y) \le a$. 
	\end{lemma}
	\begin{proof}
		We may assume that $r_M(Y) > a$. Let $B$ be a basis for $M$ containing a basis $B_Y$ for $M|Y$. We have $r(M \dcon (B-B_Y)) = r_M(Y)$, so $\tau_a(M \dcon (B-B_Y)) \le \binom{b-1}{a}^{r_M(Y)-a}$ by Theorem~\ref{kdensity}. Applying a majority argument to a smallest cover of $M \dcon (B-B_Y)$ with sets of rank at most $a$ gives a set $X' \subseteq E(M)$ so that $r_{M \dcon (B-B_Y)}(X) \le a$, and $\tau_a(M|X) \ge \binom{b-1}{a}^{a-r_M(Y)}\tau_a(M)$. Moreover, $B-B_Y$ is skew to $Y$ in $M$, so $\sqcap_M(X,Y) \le \sqcap_{M \dcon (B-B_Y)}(X,Y) \le a$.
	\end{proof}	

	\section{Thickness and Weighted Covers}
	
	The next section requires a modified notion of covering number in which elements of a cover are weighted by rank. All results in the current section are also proved in [\ref{part1}]. 
	
	  A \emph{cover} of a matroid $M$ is a collection of sets with union $E(M)$, and for an integer $d \ge 1$, we say the \textit{$d$-weight} of a cover $\cF$ of $M$ is the sum $\sum_{F \in \cF} d^{r_M(F)}$, and write $\wt^d_M(\cF)$ for this sum. Thus, a rank-$1$ set has weight $d$, a rank-$2$ set has rank $d^2$, etc. We write $\tau^d(M)$ for the minimum $d$-weight of a cover of $M$, and we say a cover of $M$ is \emph{$d$-minimal} if it has $d$-weight equal to $\tau^d(M)$. 
	
	Since $r_M(X) \le r_{M \con e}(X - \{e\}) +1$ for all $X \subseteq E(M)$, we have $\tau^d(M) \le d\tau^d(M \con e)$ for every nonloop $e$ of $M$; a simple induction argument gives the following lemma:
	
	\begin{lemma}\label{weightedcontraction}
		If $d$ is a positive integer and $M$ is a matroid, then $\tau^d(M \con C) \ge d^{-r_M(C)} \tau^d(M)$ for all $C \subseteq E(M)$.
	\end{lemma}
	
	We say a matroid $M$ is \textit{$d$-thick} if $\tau_{r(M)-1}(M) \ge d$, and a set $X \subseteq E(M)$ is \emph{$d$-thick in $M$} if $M|X$ is $d$-thick. Note that any $d$-thick matroid of rank $2$ has a $U_{2,d}$-restriction. Moreover, it is clear that $\tau_{r(M)-1}(M) \le \tau_{r(M)-2}(M \con e)$ for any nonloop $e$ of $M$, so it follows that $d$-thickness is preserved by contraction. Thus, any $d$-thick matroid of rank at least $2$ has a $U_{2,d}$-minor, and the rank-$(a+1)$ case of Theorem~\ref{kdensity} yields the following:
	
	\begin{lemma}\label{thickminor}
		Let $a,b,d$ be integers with $1 \le a < b$ and $d > \binom{b-1}{a}$. If $M$ is a $d$-thick matroid of rank greater than $a$, then $M$ has a $U_{a+1,b}$-minor. 
	\end{lemma}
	
	This controls the nature of a $d$-minimal cover of $M$ in several ways:
	
	\begin{lemma}\label{weightedcover}
		Let $a,b,d$ be integers with $1 \le a < b$ and $d > \binom{b-1}{a}$. If $\cF$ is a $d$-minimal cover of a matroid $M \in \cU(a,b)$, then
		\begin{enumerate}
		\item\label{wc1} every $F \in \cF$ is $d$-thick in $M$, §
		\item\label{wc2} every $F \in \cF$ has rank at most $a$, and 
		\item\label{wc3} $\tau_a(M) \le \tau^d(M) \le d^a \tau_a(M)$. 
		\end{enumerate}
	\end{lemma}
	\begin{proof}
		If some set $F \in \cF$ is not $d$-thick, then $F$ is the union of sets $F_1, \dotsc, F_{d-1}$ of smaller rank. Thus, $(\cF - \{F\}) \cup \{F_1, \dotsc, F_{d-1}\}$ is a cover of $M$ of weight at most $\wt_M^d(\cF) - d^{r_M(F)} + (d-1)d^{r_M(F)-1} < \wt^M_d(\cF)$, contradicting $d$-minimality of $\cF$. Therefore, every set in $F$ is $d$-thick in $M$, giving (\ref{wc1}). (\ref{wc2}) now follows from Lemma~\ref{thickminor}. 
		
		To see the upper bound in (\ref{wc3}), observe that any smallest cover of $M$ with sets of rank at most $a$ has size $\tau_a(M)$ and $d$-weight at most $d^a\tau_a(M)$. The lower bound follows from the fact that every set has $d$-weight at least $1$, and $\cF$, by (\ref{wc2}), is a $d$-minimal cover of $M$ containing sets of rank at most $a$. 
	\end{proof}

	\section{Stacking Up}
	
	Our first lemma finds, in a dense matroid, a dense minor with a large stack restriction. We consider the modified notion of density $\tau^d$. 
	
	\begin{lemma}\label{getstack}
		There is an integer-valued function $\alpha_{\ref{getstack}}(a,b,h,q,\lambda)$ so that, for any prime power $q$ and integers $a,b,h,\lambda$ with $1 \le a < b$, $m \ge 0$, and $\lambda \ge 1$, if $d > \max(q+1, \binom{b-1}{a})$ is an integer and $M \in \cU(a,b)$ satisfies $\tau^d(M) \ge \alpha_{\ref{getstack}}(a,b,h,q,\lambda)q^{r(M)}$, then $M$ has a contraction-minor $N$ with an $(h,q,a+1)$-stack restriction, satisfying $\tau^d(N) \ge \lambda q^{r(N)}$. 
	\end{lemma}
	\begin{proof}
		Let $a,b,q$ and $d$ be integers such that $1 \le a < b$, $q \ge 2$ and $d > \max(q+1,\binom{b-1}{a})$. Set $\alpha_{\ref{getstack}}(a,b,0,q,\lambda) = \lambda$, and for each $h > 0$ recursively set $\alpha_{\ref{getstack}}(a,b,h,q,\lambda) = d^{a+1}\alpha_{\ref{getstack}}(a,b,m-1,q,\lambda q^{a+1})$. Note that all values this function takes for $h > 0$ are multiples of $d$. 
		
		When $h = 0$, the lemma holds with $N = M$. Let $h > 0$ be an integer, and suppose inductively that $\alpha_{\ref{getstack}}$ as defined satisfies the lemma for smaller values of $h$. Let $M \in \cU(a,b)$ be contraction-minimal satisfying $\tau^d(M) \ge \alpha q^{r(M)}$; we show that $M$ has the required minor $N$. 
		
		\begin{claim}
			There is a set $X \subseteq E(M)$ such that $r_M(X) \le a+1$ and $M|X$ is not $\GF(q)$-representable.
		\end{claim}
		\begin{proof}[Proof of claim:]
			Let $e$ be a nonloop of $M$ and let $\cF$ and $\cF'$ be $d$-minimal covers of $M$ and $M \con e$ respectively. We consider two cases:
			
			\emph{Case 1:} $r_M(F) = 1$ for all $F \in \cF$ and $r_{M \con e}(F) = 1$ for all $ F \in \cF'$. 
			
			Note that $\tau^d(M) = d|\cF|$ and $\tau^d(M \con e) = d|\cF'|.$ By minimality of $M$, this gives $|\cF| \ge d^{-1} \alpha q^{r(M)}$ and $|\cF'| < d^{-1}\alpha q^{r(M)-1}$, so $|\cF'| \le d^{-1}\alpha q^{r(M)-1} - 1$, as this expression is an integer. Moreover, $|\cF| = \elem(M)$ and $|\cF'| = \elem(M \con e)$, so $\elem(M) \ge d^{-1}\alpha q^{r(M)} \ge q \elem(M \con e) + q > q \elem(M \con e) + 1$. Since the points of $M \con e$ correspond to lines of $M$ through $e$, it follows by a majority argument that some line $L$ through $e$ contains at least $q+1$ other points of $M$, and therefore that $X = L$ will satisfy the lemma. 
			
			\emph{Case 2:} $r_N(F) \ge 2$ for some $F \in \cF$ or $r_{M \con e}(F) \ge 2$ for some $F \in \cF'$.
			
			If $X \in \cF$ satisfies $r_M(X) \ge 2$, then by Lemma~\ref{weightedcover}, $X$ is $d$-thick in $M$ and has rank at most $a$. Since $d \ge q+2$ and thickness is preserved by contraction, the matroid $M|X$ has a $U_{2,q+2}$-minor and therefore $X$ satisfies the claim. If $X \in \cF'$ satisfies $r_{M \con e}(X) \ge 2$, then $r_M(X \cup \{e\}) \le a+1$ and $X \cup \{e\}$ will satisfy the claim for similar reasons.		

		\end{proof}
		
		Now $\tau^d(M \con X) \ge d^{-(a+1)}\tau^d(M) \ge d^{-(a+1)} \alpha q^{r(M \con X)} \ge \alpha_{\ref{getstack}}(a,b,h-1,q,\lambda q^{a+1}) q^{r(M \con X)},$ so $M \con X$ has a contraction-minor $M \con (X \cup C)$ with an $(h-1,q,a+1)$-stack restriction $S'$, satisfying $\tau^d(M') \ge \lambda q^{a+1} q^{r(M')}$. We may assume that $C$ is independent in $M \con X$; let $N = M \con C$. We have $N|X = M|X$ and $N \con X$ has an $(h-1,q,a+1)$-stack restriction, so $N$ has an $(h,q,a+1)$-stack restriction. Morever $\tau^d(N) \ge \tau^d(N \con X) \ge \lambda q^{a+1} q^{r(N \con X)} = \lambda q^{a+1-r_N(X)} q^{r(N)}$. Since $r_N(X) \le a+1$, the matroid $N$ is the required minor. 
		
	\end{proof}

	\section{Exploiting a Stack}
	
	We defined a stack as an example of a matroid that is `far' from being $\GF(q)$-representable. In this section we make this concrete by proving that a stack on top of a projective geometry yields a large uniform minor or a large projective geometry over a larger field. 
	
	We first need an easily proved lemma from [\ref{dhj}], telling us that a small projection of a projective geometry does not contain a large stack: 
	
	\begin{lemma}\label{stackinprojection}
		Let $q$ be a prime power and $h$ be a nonnegative integer. If $M$ is a matroid and $X \subseteq E(M)$ satisfies $r_M(X) \le h$ and $\si(M \del X) \cong \PG(r(M)-1,q)$, then $M \con X$ has no $(q,h+1)$-stack restriction. 
	\end{lemma}
	\begin{proof}
		The result is clear if $h=0$; suppose that $h > 0$ and that the result holds for smaller $h$. Moreover suppose for a contradiction that $M \con X$ has a $(q,h+1,t)$-stack restriction $S$. Let $F = F_1(S)$. Since $(M \con X)|F$ is not $\GF(q)$-representable but $M|F$ is, it follows that $\sqcap_{M}(F,X) > 0$. Therefore  $r_{M \con F}(X) < r_M(X) \le h$ and $\si(M \con F \del X) \cong \PG(r(M \con F)-1,q)$, so by the inductive hypothesis $M \con (X \cup F)$ has no $(q,h)$-stack restriction. Since $M \con (X \cup F)|(E(S)-F)$ is clearly such a stack, this is a contradiction.  
	\end{proof}
		
	Next we show that a large stack on top of a projective geometry guarantees (in a minor) a large flat with limited connectivity to sets in the geometry:
	
	\begin{lemma}\label{stackfindflat}
			Let $q$ be a prime power and $k \ge 0$ be an integer. If $M$ is a matroid with a $\PG(r(M)-1,q)$-restriction $R$ and a $(k^4,q)$-stack restriction, then there is a minor $M'$ of $M$ of rank at least $r(M) - k$, with a $\PG(r(M')-1,q)$-restriction $R'$ and a rank-$k$ flat $K$ such that $\sqcap_{M'}(X,K) \le \tfrac{1}{2}r_{M'}(X)$ for all $X \subseteq E(R')$.		
		\end{lemma}
		\begin{proof}
			
			Let $J \subseteq E(M)$ be maximal so that $\sqcap_M(X,J) \le \tfrac{1}{2}r_M(X)$ for all $X \subseteq E(R)$. Note that $J \cap E(R) = \varnothing$. We may assume that $r_M(J) < k$, as otherwise $J = K$ and $M' = M$ will do. Let $M' = M \con J$.
			
			\begin{claim}
				For each nonloop $e$ of $M'$, there is a set $Z_e \subseteq E(R)$ such that $r_{M'}(Z_e) \le k$ and $e \in \cl_{M'}(Z_e)$. 
			\end{claim}
			\begin{proof}[Proof of claim:]
				Let $e$ be a nonloop of $M'$. By maximality of $J$ there is some $X \subseteq E(R)$ such that $\sqcap_{M}(X,J \cup \{e\}) > \tfrac{1}{2}r_M(X)$. Let $c = \sqcap_M(X, J \cup \{e\})$, noting that $\tfrac{1}{2}r_M(X) < c \le r_M(J \cup \{e\}) \le k$. We also have $\tfrac{1}{2}r_M(X) \ge \sqcap_{M}(X,J) \ge c-1$, so $\sqcap_M(X,J) = c-1$, giving $e \in \cl_{M'}(X)$. Now $r_M(X) \le 2c-1$ and $r_{M'}(X) = r_M(X) - \sqcap_M(X,J) \le (2c-1) - (c-1) = c \le k$. Therefore $Z_e = X$ satisfies the claim. 
			\end{proof}
			
			If $e$ is not parallel in $M'$ to a nonloop of $R$, then $M'|(e \cup Z_e)$ is not $\GF(q)$-representable, as it is a simple cosimple extension of a projective geometry; this fact still holds in any contraction-minor for which $e$ is a nonloop satisfying this condition. Let $j \in \{0, \dotsc, k\}$ be maximal such that $M'$ has a $(q,j,k)$-stack restriction $T$ with the property that, for each $i \in \{1, \dotsc, j\}$, the matroid $T \con (F_1(T) \cup \dotsc \cup F_{i-1}(T))|F_i(T)$ has a basis contained in $E(R)$. For each $i$, let $F_i = F_i(T)$, and $B_i \subseteq E(R)$ be such a basis. We split into cases depending on whether $j \ge k$. 
			
			\emph{Case 1:} $j < k$.
			
			Let $M'' = M' \con E(T) = M \con (E(T)\cup J)$. If $M''$ has a nonloop $x$ that is not parallel in $M' \con E(T)$ to an element of $E(R)$, then the restriction $M''|(x \cup (Z_x-E(T)))$ has rank at most $k$, is not $\GF(q)$-representable, and has a basis contained in $Z_x \subseteq E(R)$; this contradicts the maximality of $j$. Therefore we may assume that every nonloop of $M''$ is parallel to an element of $R$, so $\si(M'') \cong \si(M|(E(R) \cup E(T) \cup J) \con (E(T) \cup J))$. We have $r_M(E(T) \cup J) \le jk+k-1 < k^2$, so by Lemma~\ref{stackinprojection} the matroid $M''$ has no $(k^2,q)$-stack restriction. However, $S$ is a $(k^4,q)$-stack restriction of $M$ and $k^4 \ge k^2 (r_M(E(T) \cup J) + 1)$, so $M''$ has a $(k^2,q)$-stack restriction by Lemma~\ref{stackrobust}. This is a contradiction. 
			
			\emph{Case 2:} $j = k$.
			
			For each $i \in \{0, \dotsc, k\}$, let $M_i = M' \con (F_1 \cup \dotsc \cup F_i)$ and $R_i = R|\cl_{R}(B_{i+1} \cup \dotsc \cup B_k)$. Note that $R_i$ is a $\PG(r(M_i)-1,q)$-restriction of $M_i$. We make a technical claim:
			\begin{claim}
				For each $i \in \{0, \dotsc, k\}$, there is a rank-$(k-i)$ independent set $K_i$ of $M_i$ so that $\sqcap_{M_i}(X,K_i) \le \tfrac{1}{2}r_{M_i}(X)$ for all $X \subseteq E(R_0) \cap E(M_i)$. 
			\end{claim}
			\begin{proof}
				When $i = k$, there is nothing to prove. Suppose inductively that $i \in \{0, \dotsc, k-1\}$ and that the claim holds for larger $i$. Let $K_{i+1}$ be a rank-$(k-i-1)$ independent set in $M_{i+1}$ so that $\sqcap_{M_{i+1}}(X,K_{i+1}) \le \tfrac{1}{2}r_{M_i}(X)$ for all $X \subseteq E(R_{0}) \cap E(M_{i+1})$. The restriction $M_i |F_{i+1}$ is not $\GF(q)$-representable; let $e$ be a nonloop of $M_i|F_{i+1}$ that is not parallel in $M_i$ to a nonloop of $R_i$. Set $K_i = K_{i+1} \cup \{e\}$, noting that $K_i$ is independent in $M_i$. Let $X \subseteq E(R_0) \cap E(M_i)$;  since $M_{i+1} = M_i \con F_{i+1}$ we have 
			\begin{align*}
				\sqcap_{M_i}(X,K_i) &= \sqcap_{M_{i+1}}(X-F_{i+1},K_i) + \sqcap_{M_i}(K_i,F_{i+1}) + \sqcap_{M_i}(X,F_{i+1}) \\ &- \sqcap_{M_i}(X \cup K_i, F_{i+1}).
			\end{align*}
			Now $e$ is a loop and $K_i - \{e\}$ is independent in $M_{i+1}$, so $\sqcap_{M_i}(K_i,F_{i+1}) = 1$, and $\sqcap_{M_{i+1}}(X-F_{i+1},K_i) = \sqcap_{M_{i+1}}(X-F_{i+1},K_{i+1}) \le \tfrac{1}{2} r_{M_{i+1}}(X) = \tfrac{1}{2}(r_{M_i}(X) - \sqcap_{M_i}(X,F_{i+1}))$. This gives
			\[\sqcap_{M_i}(X,K_i) \le \tfrac{1}{2}r_{M_i}(X) + 1 + \tfrac{1}{2}\sqcap_{M_i}(X,F_{i+1}) - \sqcap_{M_i}(X \cup K_i,F_{i+1}).\]
			It therefore suffices to show that $\sqcap_{M_i}(X \cup K_i,F_{i+1}) \ge 1 + \tfrac{1}{2}\sqcap_{M_i}(X,F_{i+1})$. Note that $e \in K_i \cap F_{i+1}$, so $\sqcap_{M_i}(X \cup K_i,F_{i+1}) \ge \max(1,\sqcap_{M_i}(X,F_{i+1}))$. Given this, it is easy to see that the inequality can only be violated if $\sqcap_{M_i}(X \cup K_i,F_{i+1}) = \sqcap_{M_i}(X,F_{i+1}) = 1$. If this is the case, then $\sqcap_{M_i}(X, B_{i+1}) = 1$ and so there is some $f \in E(R_{i+1})$ spanned by $X$ and $B_{i+1}$, since both are subsets of the projective geometry $R_{i+1}$. But $e$ and $f$ are not parallel by choice of $e$, so $\sqcap_{M_i}(X \cup K_i,F_{i+1}) \ge r_{M_i}(\{e,f\}) = 2$, a contradiction. 
		\end{proof}
		Since $r(M_0) = r(M') > r(M)-k$, taking $i = 0$ in the claim now gives the lemma. 
		\end{proof}
		Finally, we use the flat found in the previous lemma and Theorem~\ref{halfwaypoint} to find a large projective geometry minor over a larger field. 
		
		\begin{lemma}\label{stackwin}
		There is an integer-valued function $f_{\ref{stackwin}}(a,b,n,q,t)$ so that, for any prime power $q$ and integers $n,a,b$ with $n \ge 1$ and $1 \le a < b$, if $M \in \cU(a,b)$ has a $\PG(r(M)-1,q)$-restriction and an $(f_{\ref{stackwin}}(a,b,n,q,t),q,t)$-stack restriction, then $M$ has a $\PG(n-1,q')$-minor for some $q' > q$. 
		\end{lemma}
		\begin{proof}
			Let $q$ be a prime power, and $t,n,a,b$ be integers so that $t \ge 0$, $n \ge 1$, and $1 \le a < b$. Let $k \ge 2a$ be an integer so that $q^{t^{-1}r^{1/4}-2a} \ge r^{f_{\ref{halfwaypoint}}(a,b,n,q)}$ for all integers $r \ge k$. Set $f_{\ref{stackwin}}(a,b,n,q,t) = k^4$.
			 
			Let $M$ be a matroid with a $\PG(r(M)-1,q)$-restriction $R$ and a $(k^4,q,t)$-stack restriction $S$. We will show that $M$ has a $\PG(n-1,q')$-minor for some $q' > q$; we may assume (by contracting points of $R$ not spanned by $S$ if necessary) that $r(M) = r(S)$. By Lemma~\ref{stackfindflat}, there is a minor $M'$ of $M$, of rank at least $r(M)-k$, with a $\PG(r(M')-1,q)$-restriction $R'$ and a rank-$k$ flat $K$ such that $\sqcap_{M'}(K,X) \le \tfrac{1}{2}r_{M'}(X)$ for all $X \subseteq E(R')$. Let $r = r(M')$,  $M_0 = M' \con K$ and $r_0 = r(M_0)$. Since $k^4 + 2k \le 2k^4 \le r(M) \le tk^4$ and $r_0 = r-k \ge r(M)-2k$, we have 
			\[r \ge \frac{tk^4}{tk^4-k}r_0 > \left(1 + \tfrac{1}{tk^3}\right)r_0 \ge r_0 + t^{-1}(r_0)^{1/4}\]
			By choice of $k$, every rank-$a$ set in $M_0$ has rank at most $2a$ in $M'$, so $\tau_a(M_0) \ge \tau_{2a}(M')$. Moreover, a counting argument gives $\tau_{2a}(M') \ge \tau_{2a}(R') \ge \tfrac{q^r-1}{q^{2a}-1} > q^{r-2a}$, since $r > k \ge 2a$. Therefore  
			\[\tau_a(M_0) \ge \tau_{2a}(M') \ge q^{r_0 + t^{-1}(r_0)^{1/4} -2a} \ge (r_0)^{f_{\ref{halfwaypoint}}(a,b,n,q)}q^{r_0},\]
			and the result follows from Theorem~\ref{halfwaypoint}. 
	
	
	

	\end{proof}
		
	\section{Connectivity}
		
		A matroid $M$ is \emph{weakly round} if there do not exist sets $A$ and $B$ with union $E(M)$, so that $r_M(A) \le r(M)-2$ and $r_M(B) \le r(M)-1$. This is a variation on \emph{roundness}, a notion equivalent to infinite vertical connectivity introduced by Kung [\ref{kungroundness}] under the name of \emph{non-splitting}. Note that weak roundness is preserved by contractions.
		
		It would suffice in this paper to consider roundness in place of weak roundness, but we use weak roundness in order that a partial result, Lemma~\ref{maintechnicalmodified}, is slightly stronger; this should be useful in future work. 
			
		\begin{lemma}\label{roundnessreduction}
			Let $a \ge 1$ and $q \ge 2$ be integers, and $\alpha \ge 0$ be a real number. If $M$ is a matroid with $\tau_a(M) \ge \alpha q^{r(M)}$, then $M$ has a weakly round restriction $N$ such that $\tau_a(N) \ge \alpha q^{r(N)}$. 
		\end{lemma}	
		\begin{proof}
			If $r(M) \le 2$, then $M$ is weakly round, and $N = M$ will do; assume that $r(M) > 2$, and $M$ is not weakly round. There are sets $A,B \subseteq E(M)$ such that $r(M|A) < r(M)$, $r(M|B) < r(M)$ and $A \cup B = E(M)$. Now, $\tau_a(M|A) + \tau_a(M|B) \ge \tau_a(M) \ge \alpha q^{r(M)}$, so  one of $M|A$ or $M|B$ satisfies $\tau_a \ge \tfrac{1}{2}\alpha q^{r(M)} \ge \alpha q^{r(M)-1}$. The lemma follows by induction. 
		\end{proof}
		
		The way we exploit weak roundness of $M$ is to contract one restriction of $M$ into another restriction of larger rank:
		
		\begin{lemma}\label{exploitroundness}
		Let $M$ be a weakly round matroid, and $X,Y \subseteq E(M)$ be sets with $r_M(X) < r_M(Y)$. There is a minor $N$ of $M$ so that $N|X = M|X$, $N|Y = M|Y$, and $Y$ is spanning in $N$. 
		\end{lemma}
		\begin{proof}
		Let $C \subseteq E(M) - X \cup Y$ be maximal such that $(M \con C)|X = M|X$ and $(M \con C)|Y = M|Y$. The matroid $M \con C$ is weakly round, and by maximality of $C$ we have $E(M \con C) = \cl_{M \con C}(X) \cup \cl_{M \con C}(Y)$. If $r_{M \con C}(Y) < r(M \con C)$, then since $r_{M \con C}(X) \le r_{M \con C}(Y)-1$, the sets $\cl_{M \con C}(X)$ and $\cl_{M \con C}(Y)$ give a contradiction to weak roundness of $M \con C$. Therefore $Y$ is spanning in $M \con C$ and $N = M \con C$ satisfies the lemma. 
		\end{proof}
			
	\section{The Main Result}
	
	We are almost ready to prove Theorem~\ref{main}; we first prove a more technical statement from which it will follow. 
	
	\begin{lemma}\label{maintechnicalmodified}
		There is an integer-valued function $f_{\ref{maintechnicalmodified}}(a,b,n,q,t)$ so that, for any prime power $q$ and integers $a,b,n,t$ with $1 \le a < b$ and $t \ge 1$, if $M \in \cU(a,b)$ is weakly round and has an $((a+1)n,q,t)$-stack restriction and a $\PG(f_{\ref{maintechnicalmodified}}(a,b,n,q,t)-1,q)$-minor, then either $M$ has a minor $N$ with a $\PG(r(N)-1,q)$-restriction and an $(n,q,t)$-stack restriction, or $M$ has a $\PG(n-1,q')$-minor for some $q' > q$. 
	\end{lemma}
	\begin{proof}
		Let $q$ be a prime power and $a,b,n,t$ be integers with $1 \le a < b$ and $t \ge 1$. Let $d = \tbinom{b-1}{a}$ and $h = (a+1)n$. Set $f_{\ref{maintechnicalmodified}}(a,b,n,q,t)$ to be an integer $m > 2$ so that $d^{-2ht}q^{r-ht-a} \ge r^{f_{\ref{halfwaypoint}}(a,b,nt+1,q-1)}(q-1)^r$ for all integers $r \ge m/2$. 
		
		Let $M$ be a weakly round matroid with a $\PG(m-1,q)$-minor $G = M \con C \del D$ and an $(h,q,t)$-stack restriction $S$. Let $M'$ be obtained from $M$ by contracting a maximal subset of $C$ that is skew to $E(S)$; clearly $M'$ has $G$ as a minor and $r(M') \le r(G) + r(S) \le r(G) + ht$. We have $\tau_a(M') \ge \tau_a(G) \ge \tfrac{q^{r(G)}-1}{q^a-1} > q^{r(M')-ht-a}$ and $r(S) \le ht$; by Lemma~\ref{reduceconn} there is a set $X \subseteq E(M')$ such that $\tau_a(M'|X) \ge d^{a-ht}q^{r(M')-ht-a}$ and $\sqcap_{M'}(X,E(S)) \le a$. If we choose a maximal such $X$, then we have $r_{M'}(X) \ge r(M')-r(S) \ge m-ht$. 
		
		By Lemma~\ref{skewstack}, there is a set $C' \subseteq E(S)$ such that $(M' \con C')|E(S)$ has an $(n,q,t)$-stack restriction $S'$, and $E(S')$ is skew to $X$ in $M' \con C'$. By Corollary~\ref{kdensitycon}, we have 
		\[\tau_a((M' \con C')|X) \ge d^{a-ht-r_{M'}(C')}q^{r(M')-ht-a} \ge d^{-2ht}q^{r((M' \con C')|X)-ht-a},\]
		and since $r_{M' \con C'}(X) \ge r_{M'}(X) - ht \ge m-2ht \ge m/2 > 1$, it follows from Theorem~\ref{halfwaypoint} and the definition of $m$ that $(M' \con C')|X$ has a $\PG(nt,q')$-minor $G' = (M' \con C')\con C'' \del D''$  for some $q' > q-1$, where $C'' \subseteq X$. Now $M' \con (C' \cup C'')$ is a weakly round matroid with $S'$ as a restriction and $G'$ as a restriction; if $q' > q$ then we have the second outcome as $nt \ge n-1$, otherwise $q' = q$ and the first outcome follows from Lemma~\ref{exploitroundness} and the fact that $r(S') \le nt < r(G')$. 
	\end{proof}

	We now restate and prove Theorem~\ref{mainexp}, which follows routinely.
		
	\begin{theorem}\label{shortcutmain}
		There is an integer-valued function $\alpha_{\ref{shortcutmain}}(a,b,n,q)$ so that, for any integers $a,b,n$ and $q$ with $n \ge 1$, $q \ge 2$ and $1 \le a < b$, if $M \in \cU(a,b)$ satisfies $\tau_a(M) \ge \alpha_{\ref{shortcutmain}}(a,b,n,q)q^{r(M)}$, then $M$ has a $\PG(n-1,q')$-minor for some $q' > q$. 
	\end{theorem}
	\begin{proof}
		
		Let $a,b,n$ and $q$ be integers with $n \ge 1$, $q \ge 2$ and $1 \le a < b$. Let $d = \max(q,\binom{b-1}{a})+2$. Let $q^*$ be the smallest prime power so that $q^* \ge q$. Let $h = \max(n,f_{\ref{stackwin}}(a,b,n,q^*,a+1))$. Let $h' = (a+1)h$ and $m = f_{\ref{maintechnicalmodified}}(a,b,h,q,a+1)$. Let $\lambda > 0$ be an integer such that $\lambda d^{-a} q^r \ge r^{f_{\ref{halfwaypoint}}(a,b,m,q-1)}(q-1)^r$ for all integers $r \ge 1$. Set $\alpha_{\ref{shortcutmain}}(a,b,n,q) = \alpha = \max(\lambda,f_{\ref{getstack}}(a,b,h',q,\lambda))$. 
		
		Let $M \in \cU(a,b)$ satisfy $\tau_a(M) \ge \alpha q^{r(M)}$. By Theorem~\ref{halfwaypoint} and the fact that $\alpha > \lambda$, $M$ has a $\PG(m-1,q')$-minor for some $q' > q-1$; if $q' \ne q$ then we are done because $h \ge n$, so we can assume that $q = q^* = q'$. By Lemma~\ref{roundnessreduction}, $M$ has a weakly round restriction $M'$ with $\tau_a(M') \ge \alpha q^{r(M')}$. By Lemma~\ref{getstack}, $M'$ has a contraction-minor $N$ with an $(h',q,a+1)$-stack restriction, satisfying $\tau^d(N) \ge \lambda q^{r(N)}$. We have $\tau_a(N) \ge d^{-a}\tau^d(N) \ge d^{-a}\lambda q^{r(N)}$, so by definition of $\lambda$ the matroid $N$ has a $\PG(m-1,q')$-minor for some $q'' > q-1$. As before, we may assume that $q'' = q$. By Lemma~\ref{maintechnicalmodified} and the definitions of $h'$ and $m$, we may assume that there is a minor $N'$ of $N$ with a $\PG(r(N')-1,q)$-restriction and an $(h,q,a+1)$-stack restriction. The result now follows from Lemma~\ref{stackwin}. 
		\end{proof}
	
	Theorem~\ref{main} is a fairly simple consequence. 
	
	\begin{theorem}\label{mainrep}
			If $a \ge 1$ is an integer, and $\cM$ is a minor-closed class of matroids, then there is an integer $c$ so that either:
		\begin{enumerate}
			\item\label{mr1} $\tau_a(M) \le r(M)^{c}$ for all $M \in \cM$, or
			\item\label{mr2} There is a prime power $q$ so that $\tau_a(M) \le c q^{r(M)}$ for all $M \in \cM$ and $\cM$ contains all $\GF(q)$-representable matroids, or
			\item\label{mr3} $\cM$ contains all rank-$(a+1)$ uniform matroids. 
		\end{enumerate}
	\end{theorem}
	\begin{proof}
		We may assume that~(\ref{mr3}) does not hold; let $b > a$ be an integer such that $\cM \subseteq \cU(a,b)$. As $U_{a+1,b}$ is a simple matroid that is $\GF(q)$-representable whenever $q \ge b$ (see [\ref{hirschfeld}]), we have $\PG(a,q') \notin \cM$ for all $q' \ge b$.
		
		If, for some integer $n > a$, we have $\tau_a(M) < r(M)^{f_{\ref{mainpoly}}(a,b,n)}$ for all $M \in \cM$ of rank at least $2$, then~(\ref{mr1}) holds. We may therefore assume that, for all $n > a$, there exists a matroid $M_n \in \cM$ such that $r(M_n) \ge 2$ and $\tau_a(M_n) \ge r(M_n)^{f_{\ref{mainpoly}}(a,b,n)}$. 
		
		By Theorem~\ref{mainpoly}, it follows that for all $n > a$ there exists a prime power $q'_n$ such that $\PG(n-1,q'_n) \in \cM$. We have $q'_n < b$ for all $n$, so there are finitely many possible $q'_n$, and so there is a prime power $q_0 < b$ such that $\PG(n-1,q_0) \in \cM$ for infinitely many $n$, implying that $\cM$ contains all $\GF(q_0)$-representable matroids.
		
		Let $q$ be maximal such that $\cM$ contains all $\GF(q)$-representable matroids. Since $\PG(a,q') \notin \cM$ for all $q' \ge b$, the value $q$ is well-defined, and moreover there is some $n$ such that $\PG(n-1,q') \notin \cM$ for all $q' > q$. Theorem~\ref{mainexp} thus gives $\tau_a(M) \le \alpha_{\ref{mainexp}}(a,b,n,q) q^{r(M)}$ for all $M\in \cM$, giving~(\ref{mr2}). 
	\end{proof}
	
	Finally, we prove a theorem that will be used in a future paper to obtain stronger results. This theorem may appear to follow directly from Lemmas~\ref{stackwin} and~\ref{maintechnicalmodified}, but is subtly stronger; the independence of the function $h$ on $t$ requires Theorem~\ref{shortcutmain} to be obtained. 
		
	\begin{theorem}\label{portable}
		There are integer-valued functions $f_{\ref{portable}}(a,b,n,q,t)$ and $h_{\ref{portable}}(a,b,n,q)$ so that, for every prime power $q$ and all positive integers $a,b,n,t$ with $a < b$, if $M \in \cU(a,b)$ is weakly round with a $\PG(f_{\ref{portable}}(a,b,n,q,t)-1,q)$-minor and an $(h_{\ref{portable}}(a,b,n,q),q,t)$-stack restriction, then $M$ has a $\PG(n-1,q')$-minor for some $q' > q$. 
	\end{theorem}
	\begin{proof}
		Let $q$ be a prime power and $a,b,n,t$ be positive integers with $a < b$. Let $\alpha = \alpha_{\ref{shortcutmain}}(a,b,n,q)$ and let $k$ be an integer such that $q^{k-2a-1} \ge \alpha$. Set $h_{\ref{portable}}(a,b,n,q) = h = \max(n,(a+1)k^4)$ and $f_{\ref{portable}}(a,b,n,q,t) = m = f_{\ref{maintechnicalmodified}}(a,b,h,q,t)$.  
		
		Let $M \in \cU(a,b)$ be weakly round with a $\PG(m-1,q)$-minor and an $(h,q,t)$-stack restriction $S$. By Lemma~\ref{maintechnicalmodified} and the fact that $h \ge n$, we may assume that $M$ has a minor $N$ with a $\PG(r(N)-1,q)$-restriction and a $(k^4,q,t)$-stack restriction. By Lemma~\ref{stackfindflat}, there is a minor $N'$ of $N$ with a $\PG(r(N')-1,q)$-restriction $R'$ and a rank-$k$ flat $K$ such that $\sqcap_{N'}(X,K) \le \tfrac{1}{2}r_{M'}(X)$ for all $X \subseteq E(R')$. It follows that $\tau_a(N' \con K) \ge \tau_{2a}(R') \ge \frac{q^{r(R')}-1}{q^{2a}-1} \ge q^{k-2a-1}q^{r(N' \con K)} \ge \alpha q^{r(N' \con K)}$, and Theorem~\ref{shortcutmain} gives the result. 
	\end{proof}
	
\section*{References}
\newcounter{refs}

		\begin{list}{[\arabic{refs}]}
{\usecounter{refs}\setlength{\leftmargin}{10mm}\setlength{\itemsep}{0mm}}

\item\label{openprobs}
J. Geelen, 
Some open problems on excluding a uniform matroid, 
Adv. in Appl. Math. 41(4) (2008), 628--637.

\item\label{gk}
J. Geelen, K. Kabell,
Projective geometries in dense matroids, 
J. Combin. Theory Ser. B 99 (2009), 1--8.

\item\label{gkb}
J. Geelen, K. Kabell, 
The {E}rd{\H o}s-{P}\'osa property for matroid circuits,
J. Combin. Theory Ser. B 99 (2009), 407--419.  

\item\label{gkw}
J. Geelen, J.P.S. Kung, G. Whittle,
Growth rates of minor-closed classes of matroids,
J. Combin. Theory. Ser. B 99 (2009), 420--427.

\item\label{part1}
J. Geelen, P. Nelson, 
Projective geometries in exponentially dense matroids. I, 
Submitted.

\item\label{dhj}
J. Geelen, P. Nelson, 
A density Hales-Jewett theorem for matroids, 
Submitted. 

\item\label{gw}
J. Geelen, G. Whittle,
Cliques in dense $\GF(q)$-representable matroids, 
J. Combin. Theory. Ser. B 87 (2003), 264--269.

\item\label{hirschfeld}
J. W. P. Hirschfeld,
Complete Arcs, 
Discrete Math. 174(1-3):177--184 (1997),
Combinatorics (Rome and Montesilvano, 1994).

\item\label{kungroundness}
J.P.S. Kung, 
Numerically regular hereditary classes of combinatorial geometries,
Geom. Dedicata 21 (1986), no. 1, 85--10.

\item\label{thesis}
P. Nelson,
Exponentially Dense Matroids,
Ph.D thesis, University of Waterloo (2011). 

\item \label{oxley}
J. G. Oxley, 
Matroid Theory,
Oxford University Press, New York (2011).
\end{list}

\end{document}